\documentclass[preprint,12pt]{elsarticle}

\usepackage{amssymb}
\usepackage{amsthm}
\usepackage[utf8]{inputenc}
\usepackage[T1]{fontenc}

\newtheorem{proposition}{Proposition}[section]

\theoremstyle{definition}
\newtheorem{definition}{Definition}[section]

\theoremstyle{remark}
\newtheorem{remark}{Remark}[section]

\usepackage{mathtools}
\usepackage{hyperref}
\usepackage{bm}
\usepackage{bbm}
\usepackage{tikz}
\usepackage{graphicx}

\journal{Examples and Counterexamples}

\begin{document}

\begin{frontmatter}


\title{Nonuniqueness of lattice Boltzmann schemes derived from finite difference methods}

\author[IANM,LBRG]{Eliane Kummer\corref{cor1}}
\affiliation[IANM]{
            organization={Institute of Applied and Numerical Mathematics, Karlsruhe Institute of Technology},
            addressline={Englerstr.~2},  
            city={Karlsruhe},
            postcode={76131},
            country={Germany}
            }
\affiliation[LBRG]{
            organization={Lattice Boltzmann Research Group, Karlsruhe Institute of Technology},
            addressline={Englerstr.~2},  
            city={Karlsruhe},
            postcode={76131},
            country={Germany}
            }    

\author[IANM,LBRG]{Stephan Simonis}
\cortext[cor1]{Corresponding author, Email: eliane.kummer@student.kit.edu}

\begin{abstract}
Recently, the construction of finite difference schemes from lattice Boltzmann schemes has been rigorously analyzed [Bellotti \textit{et al.}~(2022), Numer.\ Math.\ 152, pp.\ 1--40]. 
It is thus known that any lattice Boltzmann scheme can be expressed in terms of a corresponding multi-step finite difference scheme on the conserved variables. 
In the present work, we provide counterexamples for the conjecture that any multi-step finite difference scheme has a unique lattice Boltzmann formulation. 
Based on that, we indicate the existence of equivalence classes for discretized relaxation systems. 
\end{abstract}


\end{frontmatter}


\section{Introduction}

Hilbert's sixth problem \cite{hilbert1902mathematical} is based on the idea of obtaining a more mathematical approach to physics than previously existent ones \cite{gorban2014hilbert,gorban2018hilbert}. 
Meanwhile, by axiomatization of formal physics, the analysis on existing physical theories has helped, for example in fluid mechanics, to rigorously move from the mesoscopic scale based on kinetic theory, to macroscopic continuum laws, usually in the form of classical partial differential equations. 
Specifically, solutions of the mesoscopic Boltzmann equation with Bhatnagar--Gross--Krook (BGK) collision \cite{bhatnagar1954model}, which abstracts the particle distribution of a fluid, have been linked to weak solutions of macroscopic incompressible Navier--Stokes equations \cite{saint-raymond2003bgk}. 
From an abstract point of view, the BGK collision can be described as relaxation, see for example \cite{bouchut2000diffusive,aregba-driollet2000discrete,yong2002basic} and references therein. 
For scalar, linear advection--diffusion equations in \(d\) dimensions, Simonis \textit{et al.}~\cite{simonis2020relaxation,simonis2022constructing,simonis2023lattice} have proved that deriving lattice Boltzmann methods in a top-down fashion---by construction of relaxation systems of BGK-type---is possible. 
The methodology is based on systematically building up a moment system structure in the relaxation system. 
Artificial variables (identified as moments) are subsequently introduced, which yields a specific form of the constructed relaxation system. 
However, parameter choices within this construction are mostly nonunique. 
From a conceptual perspective, this nonuniqueness is a valid observation, since uniquely constructable relaxation systems or schemes for given partial differential equations would imply a unique inverse path to Hilbert's model hierarchy. 
On the discrete level, Bellotti \textit{et al.}~\cite{bellotti2022finite} rigorously analyzed the formulation of a class of lattice Boltzmann schemes (LBSs) in terms of multi-step finite difference operators. 
Moreover, Bellotti~\cite{bellotti2022rigorous} observed that the reverse action (formulating finite difference schemes (FDSs) in terms of lattice Boltzmann equations) is not unique, which again matches the above-mentioned observation. 
The fact that different LBSs or akin schemes can yield the same FDS has also been indicated earlier by Dellacherie~\cite{dellacherie2014construction}. 
Notably, even if the FDSs derived from two distinct LBSs are the same, the behavior of these schemes at initialization may differ due to the mesoscopic disparity \cite{bellotti2024initialisation}, \textit{i.e.} the presence of nonconserved moments.

In this work, we provide explicit counterexamples for the statement that any multi-step FDS has a unique lattice Boltzmann formulation. 
Based on that, we demonstrate the nonuniqueness of equivalent FDSs for lattice Boltzmann methods, and indicate the existence of equivalence classes for discretized relaxation systems.

The notation is introduced in Section~\ref{sec:methodology}, the counterexamples are constructed and discussed in Section~\ref{sec:counterexample}, and the paper is closed with promising future research directions in Section~\ref{sec:conclusions}.

\section{Methodology}\label{sec:methodology}
We first recall the definitions and the proposition by Bellotti \textit{et al.}~\cite{bellotti2022finite} linking any LBS to a corresponding FDS. 

\begin{definition}[Lattice Boltzmann scheme (LBS)]\label{def:latticeBoltzmannScheme}
Let $\mathbf{M} \in \mathrm{GL}_{q}(\mathbb{R})$ be the moment matrix and $\mathbf{S}= \mathrm{diag}(0, \dots, 0, s_{N+1}, \dots, s_{q})$ the relaxation matrix where the first $N \in \mathbb{N}$ entries are zero and the other $q-N$ entries are such that $s_{i} \in (0,2]$ for $i \in \{N+1, \dots ,q\}$.
For 
\begin{align}
	\mathbf{T} &\coloneqq \mathbf{M} \,\mathrm{diag}\left(\mathsf{T}_{\triangle x}^{\bm{c}_{1}}, \dots, \mathsf{T}_{\triangle x}^{\bm{c}_{q}}\right) \mathbf{M}^{-1} \in \mathcal{M}_{q} \left(\mathcal{D}_{\triangle x}^{d}\right), \label{eq:Tmatrix} \\
\mathbf{A} &\coloneqq \mathbf{T} \left(\mathbf{I} -\mathbf{S} \right) \in \mathcal{M}_{q} \left(\mathcal{D}_{\triangle x}^{d}\right), \label{eq:Amatrix} \\
\mathbf{B} &\coloneqq \mathbf{T}\mathbf{S} \in \mathcal{M}_{q} \left(\mathcal{D}_{\triangle x}^{d}\right), \label{eq:Bmatrix}
\end{align}
the LBS equation can be formulated as 
\begin{align}\label{eq:latticeBoltzmannScheme}
    \bm{m}^{n+1}(\bm{x}) = \mathbf{A} \bm{m}^{n}(\bm{x})+ \mathbf{B} \bm{m}^{\mathrm{eq}}|^{n}(\bm{x}) \qquad \forall \bm{x} \in \mathcal{L}.
\end{align}
The notation above is as follows: \(\mathbf{I}\) is the identity matrix, \(\mathcal{M}_{q}\left( \mathcal{D}^{d}_{\triangle x}\right)\) is the set of square matrices of size \(q\) with entries in the set of finite difference operators \(\mathcal{D}^{d}_{\triangle x}\) (see \cite[Definition 3]{bellotti2022finite}) on the discrete lattice \(\mathcal{L} = \triangle x \mathbb{Z}^{d}\), \(\mathsf{T}^{\bm{z}}_{\triangle x}\) denotes a shift operator for \(\bm{z} \in \mathbb{Z}^{d}\) on \(\mathcal{L}\) (see \cite[Definition 1]{bellotti2022finite}), \(\bm{m}^{\mathrm{eq}}\colon \mathbb{R}^{N} \to \mathbb{R}^{q} \) are the equilibria of the moments \(\bm{m} \in \mathbb{R}^{q}\). 
\end{definition}

Let \(\mathbf{A}_{i}  \coloneqq \mathbf{A}_{\{i\} \cup \{N+1, N+2, \ldots , q \}}\) which maps the rows and columns of the corresponding indices in \(\mathbf{A}\) on a \(q\times q\) matrix filled with zeros, and let \(\mathbf{A}^{\diamond} \coloneqq \mathbf{A} - \mathbf{A}_{i}\). 
Based on Definition~\ref{def:latticeBoltzmannScheme}, it has been proven in \cite{bellotti2022finite} that each LBS corresponds to an FDS which is unique for \(N=1\) but nonunique for \(N>1\) due to the nonunique elimination of nonconserved moments by admissible additive splittings \(\mathbf{A} = \mathbf{A}_{i} + \mathbf{A}_{i}^{\diamond}\) \cite{bellotti2023numerical}.

\begin{proposition}[Corresponding FDS] \label{prop:lbsToFds}
Bellotti \textit{et al.}~\cite[Proposition 6]{bellotti2022finite}: 
Let the number of conserved moments be \(N\geq 1\). 
The LBS~\eqref{eq:latticeBoltzmannScheme}, rewritten for \(i=1,2,\ldots,N\) as 
\begin{align}
\bm{m}^{n+1} = \mathbf{A}_{i} \bm{m}^{n} + \mathbf{A}_{i}^{\diamond} \bm{m}^{n} + \mathbf{B}\bm{m}^{\mathrm{eq}}\vert^{n}, 
\end{align}
corresponds to a multi-step explicit FDS of \(m_{1}, m_{2}, \ldots, m_{N}\) of the form 
\begin{align}\label{eq:finiteDiffDef}
        m_{i}^{n+1} 
        =
         & - \sum_{k=0}^{q-N}\gamma_{i,k}m_{i}^{n-q+k+N} + \left(\sum_{k=0}^{q-N}\left(\sum_{l=0}^{k}\gamma_{i,q+1-N+l-k}\mathbf{A}_{i}^{l} \right)\mathbf{A}_{i}^{\diamond} \bm{m}^{n-k}\right)_{i} \nonumber\\
         &+ \left(\sum_{k=0}^{q-N}\left(\sum_{l=0}^{k}\gamma_{i,q+1-N+l-k} \mathbf{A}_{i}^{l}\right)  \mathbf{B} \bm{m}^{\mathrm{eq}}|^{n-k}\right)_{i} ,  
\end{align}
for \(i = 1, 2, \ldots, N\), where \( ( \gamma_{i, k} )^{k=q+1-N}_{k=0} \subset \mathcal{D}_{\triangle x}^{d}\) are the coefficients of the characteristic polynomial \(\chi_{\mathbf{A}_{i}} = X^{N-1} \sum_{k=0}^{k=q+1-N} \gamma_{i,k} X^{k}\) of \(\mathbf{A}_{i}\). 
\end{proposition}

\section{Main result and counterexamples}\label{sec:counterexample}

We provide a counterexample to prove that Proposition~\ref{prop:lbsToFds} is not uniquely reversible, that is, one cannot uniquely recover a LBS from a given FDS of the form \eqref{eq:finiteDiffDef}. 
Subsequently, we illustrate our main result with several remarks and additional counterexamples.

\begin{proposition}[Corresponding LBS]\label{prop:fdsTolbs}
	From an FDS of the form \eqref{eq:finiteDiffDef}, the complete matrices \(\mathbf{A}\) and \(\mathbf{B}\) for an LBS~\eqref{eq:latticeBoltzmannScheme} cannot be uniquely recovered. 
\end{proposition}
\begin{proof}
    We prove the proposition by assuming that each FDS uniquely corresponds to an LBS and providing a counterexample to said claim. 
    The counterexample is based on finding two different LBSs, meaning two different matrices $\mathbf{A}$, $\widetilde{\mathbf{A}}$ and corresponding $\mathbf{B}$, $\widetilde{\mathbf{B}}$ leading to the same FDS. 
    Consider an LBS on $\text{D}_1\text{Q}_3$ (one dimension and three discrete velocities), with only one conserved moment, $N=1$. 
    Consider the convention \eqref{eq:Tmatrix}, \eqref{eq:Amatrix}, \eqref{eq:Bmatrix} introduced by \cite{bellotti2022finite}. 
    From Proposition~\ref{prop:lbsToFds} we know that each LBS corresponds uniquely to an FDS of the form
\begin{align}
    m_{1}^{n+1} 
    = 
    & - \sum_{k=0}^{2} \gamma_{k}m_{1}^{n+k-2} 
      + \left( \sum_{k=0}^{2} \left( \sum_{l=0}^{k} \gamma_{3+l-k}\mathbf{A}^{l} \right) \mathbf{B}\bm{m}^{\mathrm{eq}}|^{n-k}\right)_{1}  \\
    = 
    & - \sum_{k=0}^{2} \gamma_{k} m_{1}^{n+k-2} 
      + \Bigg( \sum_{l=0}^{0} \left( \gamma_{3+l}\mathbf{A}^{l} \right) \mathbf{B} \bm{m}^{\mathrm{eq}}|^{n} \nonumber \\
    & + \sum_{l=0}^{1} \left( \gamma_{3+l-1}\mathbf{A}^{l} \right)\mathbf{B}\bm{m}^{\mathrm{eq}}|^{n-1} 
      + \sum_{l=0}^{2} \left( \gamma_{3+l-2}\mathbf{A}^{l} \right)\mathbf{B} \bm{m}^{\mathrm{eq}}|^{n-2} \Bigg)_{1} \\
    = 
    & - \gamma_{0}m_{1}^{n-2} 
      - \gamma_{1}m_{1}^{n-1} 
      - \gamma_{2}m_{1}^{n}  
      + \big(   \gamma_{3} \mathbf{B} \bm{m}^{\mathrm{eq}}|^{n} 
                + \left( \gamma_{2}\mathbf{I} + \gamma_{3}\mathbf{A} \right) \mathbf{B} \bm{m}^{\mathrm{eq}}|^{n-1} \nonumber \\ 
    & + (\gamma_{1}\mathbf{I} + \gamma_{2}\mathbf{A} + \gamma_{3} \mathbf{A}^{2})\mathbf{B}\bm{m}^{\mathrm{eq}}|^{n-2} \big)_{1}. \label{eq:counterFds}
\end{align}
Since we have $\gamma_{3}=1$, we obtain
\begin{align} \label{eq:reducedFDS}
    m_{1}^{n+1} 
    = 
    & - \gamma_{0}m_{1}^{n-2} 
      - \gamma_{1}m_{1}^{n-1} 
      - \gamma_{2}m_{1}^{n} 
      + \big(   \mathbf{B} \bm{m}^{\mathrm{eq}}|^{n} 
              + \left( \gamma_{2}\mathbf{I} + \mathbf{A} \right) \mathbf{B} \bm{m}^{\mathrm{eq}}|^{n-1} \nonumber \\ 
    & + (\gamma_{1}\mathbf{I} + \gamma_{2}\mathbf{A} +  \mathbf{A}^{2})\mathbf{B}\bm{m}^{\mathrm{eq}}|^{n-2} \big)_{1}. 
\end{align}
Let \(\mathbf{A}\) and \(\mathbf{B}\) determine the first LBS realization as described above. 
We now want to find different matrices $\widetilde{\mathbf{A}} \neq \mathbf{A}$ as well as $\widetilde{\mathbf{B}} \neq \mathbf{B}$ that fulfill this FDS~\eqref{eq:reducedFDS} on \(m_{1}\), since then we obtain $\widetilde{\mathbf{T}}(\mathbf{I}-\mathbf{S})  \neq \mathbf{T}(\mathbf{I}-\mathbf{S})$, which in turn implies $\widetilde{\mathbf{T}} \neq \mathbf{T}$. 
From this it follows that
\begin{align}
    \widetilde{\mathbf{M}} \mathrm{diag} \left( \mathsf{T}_{\triangle x}^{\bm{c}_{1}}, \dots, \mathsf{T}_{\triangle x}^{\bm{c}_{q}} \right) \widetilde{\mathbf{M}}^{-1}  \neq \mathbf{M} \mathrm{diag} \left(\mathsf{T}_{\triangle x}^{\bm{c}_{1}}, \dots, \mathsf{T}_{\triangle x}^{\bm{c}_{q}}\right) \mathbf{M}^{-1} .
    \label{eq:lastNeq}
\end{align}
In turn, \eqref{eq:lastNeq} implies that $\widetilde{\mathbf{M}} \neq \mathbf{M}$ and therefore we obtain two LBSs with differing moment matrices.

Let $s_{2}=s_{3}=2$ and $s_{1}=0$ (one conserved moment) define the relaxation matrix and let a moment transformation be compressed in \eqref{eq:Tmatrix}, respectively reading
\begin{align}
    \mathbf{S} 
    \coloneqq 
    \begin{bmatrix}
        0 & 0 & 0 \\
        0 & 2 & 0 \\
        0 & 0 & 2
    \end{bmatrix} \in \mathcal{M}_{3}(\mathbb{R}), 
    \qquad 
    \mathbf{T} 
     \coloneqq 
     \begin{bmatrix}
        t_{11} & t_{12} & t_{13} \\
        t_{21} & t_{22} & t_{23} \\
        t_{31} & t_{32} & t_{33}
    \end{bmatrix} \in \mathcal{M}_{3} \left(\mathcal{D}_{\triangle x}^{1}\right),
\end{align} 
with $t_{ij} \in \mathcal{D}_{\triangle x}^{1} $ and $i,j \in \{1, 2, 3\}$. 
Then we obtain from \eqref{eq:Amatrix} and \eqref{eq:Bmatrix}, respectively that 
\begin{align}
    \mathbf{A} 
    &\coloneqq 
   \begin{bmatrix}
        t_{11} & - t_{12} & - t_{13} \\
        t_{21} & - t_{22} & - t_{23} \\
        t_{31} & - t_{32} & - t_{33}
    \end{bmatrix}, \qquad
    \mathbf{B} 
    \coloneqq
    \begin{bmatrix}
        0 & 2 t_{12} & 2 t_{13} \\
        0 & 2 t_{22} & 2 t_{23} \\
        0 & 2 t_{32} & 2 t_{33}
    \end{bmatrix},
\end{align} where $\mathbf{A}, \mathbf{B} \in \mathcal{M}_{3} \left(\mathcal{D}_{\triangle x}^{1}\right)$.
If the following conditions hold for two different matrices $\mathbf{A}$ and $\widetilde{\mathbf{A}}$ and $\mathbf{B}$ and $\widetilde{\mathbf{B}}$, we obtain the same FDS.
Let in the following $\widetilde{\cdot}_{ij} \in \mathcal{D}_{\triangle x}^{1}$ with $i,j \in \{1, 2, 3\}$ denote the entries of \(\widetilde{\mathbf{T}}\) contained in $\widetilde{\mathbf{A}}$ and $\widetilde{\mathbf{B}}$, respectively.
\begin{enumerate}
    \item $\left ( \mathbf{B}\right )_{1} = \left ( \widetilde{\mathbf{B}}\right )_{1}$. 
        This is equivalent to 
        \begin{align}\label{Nontrivial Case Condition 1}
            t_{12} = \widetilde{t}_{12}, \qquad t_{13} = \widetilde{t}_{13}.
        \end{align}
    \item $\left ( \mathbf{A}\mathbf{B}\right )_{1} = \left (\widetilde{\mathbf{A}}\widetilde{\mathbf{B}}\right )_{1}$. 
        Since we have that
        \begin{align}
            \left ( \mathbf{A}\mathbf{B}\right )_{1} = 2
            \begin{pmatrix}
                0 & t_{11}t_{12} - t_{12}t_{22} - t_{13}t_{32} & t_{11}t_{13} - t_{12}t_{23} - t_{13}t_{33} 
            \end{pmatrix} , 
        \end{align} 
        this condition is equivalent to
        \begin{align}\label{Nontrivial Case Condition 2}
            t_{11}t_{12}-t_{12}t_{22} - t_{13}t_{32} &= \widetilde{t}_{11}\widetilde{t}_{12}-\widetilde{t}_{12}\widetilde{t}_{22} - \widetilde{t}_{13}\widetilde{t}_{32}, \\
            \label{Nontrivial Case Condition 3}
            t_{11}t_{13}-t_{12}t_{23} - t_{13}t_{33} &=
            \widetilde{t}_{11}\widetilde{t}_{13}-\widetilde{t}_{12}\widetilde{t}_{23} - \widetilde{t}_{13}\widetilde{t}_{33}.
        \end{align}
    \item $\left ( \mathbf{A}^{2}\mathbf{B}\right )_{1} = \left ( \widetilde{\mathbf{A}}^{2}\widetilde{\mathbf{B}}\right )_{1}$. 
        Since we have that
        \begin{align}
            \left ( \mathbf{A}^{2}\mathbf{B} \right )_{1} =  2
            \begin{pmatrix}
                0 & \sum_{i=1}^{3} \mathfrak{a}_{i} t_{i2} & \sum_{i=1}^{3} \mathfrak{a}_{i} t_{i3} 
            \end{pmatrix} ,
        \end{align} 
        where 
        \begin{align}
            \mathfrak{a}_{1} &= t_{11}^{2} - t_{12}t_{21} - t_{13}t_{31}, \\
            \mathfrak{a}_{2} &= - t_{11}t_{12} + t_{12}t_{22} + t_{13}t_{32}, \\
            \mathfrak{a}_{3} &= -t_{11}t_{13} + t_{12}t_{23} + t_{13}t_{33}, 
        \end{align}
        this condition is equivalent to
            \begin{align}
                & t_{12}(t_{11}^{2} - t_{12}t_{21} - t_{13}t_{31}) + t_{22}(- t_{11}t_{12} + t_{12}t_{22} + t_{13}t_{32}) 
                \nonumber\\
                & + t_{32}(-t_{11}t_{13} + t_{12}t_{23} + t_{13}t_{33}) 
                \nonumber\\
                & = \widetilde{t}_{12}(\widetilde{t}_{11}^{2} - \widetilde{t}_{12}\widetilde{t}_{21} - \widetilde{t}_{13}\widetilde{t}_{31}) + \widetilde{t}_{22}(- \widetilde{t}_{11}\widetilde{t}_{12} + \widetilde{t}_{12}\widetilde{t}_{22} + \widetilde{t}_{13}\widetilde{t}_{32}) 
                \nonumber \\
                & + \widetilde{t}_{32}(-\widetilde{t}_{11}\widetilde{t}_{13} + \widetilde{t}_{12}\widetilde{t}_{23} + \widetilde{t}_{13}\widetilde{t}_{33}) , \label{Nontrivial Case Condition 4} \\
                & t_{13}(t_{11}^{2} - t_{12}t_{21} - t_{13}t_{31}) + t_{23}(- t_{11}t_{12} + t_{12}t_{22} + t_{13}t_{32}) 
                \nonumber \\
                & + t_{33}(-t_{11}t_{13} + t_{12}t_{23} + t_{13}t_{33}) 
                \nonumber \\
                & = \widetilde{t}_{13}(\widetilde{t}_{11}^{2} - \widetilde{t}_{12}\widetilde{t}_{21} - \widetilde{t}_{13}\widetilde{t}_{31}) + \widetilde{t}_{23}(- \widetilde{t}_{11}\widetilde{t}_{12} + \widetilde{t}_{12}\widetilde{t}_{22} + \widetilde{t}_{13}\widetilde{t}_{32}) 
                \nonumber \\
                & + \widetilde{t}_{33}(-\widetilde{t}_{11}\widetilde{t}_{13} + \widetilde{t}_{12}\widetilde{t}_{23} + \widetilde{t}_{13}\widetilde{t}_{33}). \label{Nontrivial Case Condition 5}
             \end{align}
    \item Lastly, if the characteristic polynomials of $\mathbf{A}$ and $\widetilde{\mathbf{A}}$ coincide, we then obtain the same FDS.
    For $\mathbf{A}$ we obtain
    \begin{align}
        \chi_{\mathbf{A}} = \; &  \det(\mathbf{A} - \lambda \mathbf{I}) 
        \\
        = \; & - \lambda^3 - (-t_{11} + t_{22} + t_{33})\lambda^2 
        \nonumber
        \\
        & + (t_{11}t_{22} + t_{11}t_{33} - t_{22}t_{33} + t_{12}t_{21} + t_{13}t_{31} + t_{23}t_{32}) \lambda
        \nonumber
        \\
        & + \big[t_{13}t_{32}t_{21} + t_{12}t_{23}t_{31} + t_{13}t_{31}t_{22} 
        \nonumber
        \\
        & + t_{12}t_{21}t_{33} - t_{23}t_{32}t_{11} + t_{11}t_{22}t_{33}\big ] .
    \end{align}
    The coefficients of the normalized characteristic polynomial read
    \begin{align}
        \gamma_{3} 
        = \; & 1 , 
        \\
        \gamma_{2} 
        = \; & - t_{11} + t_{22} + t_{33} = - \mathrm{tr}(\mathbf{A}) , 
        \\
        \gamma_{1} 
        = \; & - \left ( (t_{11}t_{22} + t_{11}t_{33} - t_{22}t_{33} + t_{12}t_{21} + t_{13}t_{31} + t_{23}t_{32}) \right ) , 
        \\
        \gamma_{0}
        = \; & - \big[t_{13}t_{32}t_{21} + t_{12}t_{23}t_{31} + t_{13}t_{31}t_{22} 
        \nonumber
        \\
        & + t_{12}t_{21}t_{33} - t_{23}t_{32}t_{11} + t_{11}t_{22}t_{33}\big ]. 
    \end{align} 
    This means that if the following three conditions hold, then the characteristic polynomials of $A$ and $\widetilde{A}$ coincide
    \begin{align}
        & \mathrm{tr}(\mathbf{A}) = \mathrm{tr}(\widetilde{\mathbf{A}}), \label{Nontrivial Case Condition 6}\\
        & t_{11}t_{22} + t_{11}t_{33} - t_{22}t_{33} + t_{12}t_{21} + t_{13}t_{31} + t_{23}t_{32} \nonumber \\
        & = \widetilde{t}_{11}\widetilde{t}_{22} + \widetilde{t}_{11}\widetilde{t}_{33} - \widetilde{t}_{22}\widetilde{t}_{33} + \widetilde{t}_{12}\widetilde{t}_{21} + \widetilde{t}_{13}\widetilde{t}_{31} + \widetilde{t}_{23}\widetilde{t}_{32} , \label{Nontrivial Case Condition 7} \\
        & t_{13}t_{32}t_{21} + t_{12}t_{23}t_{31} + t_{13}t_{31}t_{22} + t_{12}t_{21}t_{33} - t_{23}t_{32}t_{11} + t_{11}t_{22}t_{33} \nonumber \\
        & = \widetilde{t}_{13}\widetilde{t}_{32}\widetilde{t}_{21} + \widetilde{t}_{12}\widetilde{t}_{23}\widetilde{t}_{31} + \widetilde{t}_{13}\widetilde{t}_{31}t_{22}  + \widetilde{t}_{12}\widetilde{t}_{21}\widetilde{t}_{33} - \widetilde{t}_{23}\widetilde{t}_{32}\widetilde{t}_{11} + \widetilde{t}_{11}\widetilde{t}_{22}\widetilde{t}_{33}.             \label{Nontrivial Case Condition 8}
    \end{align}
\end{enumerate}
    Thus, if for two different matrices $\mathbf{T}$ and $\widetilde{\mathbf{T}}$ the eight conditions \eqref{Nontrivial Case Condition 1}, \eqref{Nontrivial Case Condition 2}, \eqref{Nontrivial Case Condition 3}, \eqref{Nontrivial Case Condition 4}, \eqref{Nontrivial Case Condition 5}, \eqref{Nontrivial Case Condition 6},  \eqref{Nontrivial Case Condition 7}, and \eqref{Nontrivial Case Condition 8} hold, then we arrive at the same FDS, which proves the claim. 
\end{proof}

\begin{remark}[Equivalence classes]\label{rem:equivalence}
    To produce the same FDS, $\mathbf{A}$ and $\widetilde{\mathbf{A}}$ must have the same characteristic polynomial, so $\chi_{\mathbf{A}} = \chi_{\widetilde{\mathbf{A}}}$. 
    This is true for all matrices $\mathbf{A}$ and $\widetilde{\mathbf{A}}$ which are similar to one another. 
    The latter is the case, if $\mathbf{A}$ and $\widetilde{\mathbf{A}}$ lie in the same equivalence class. 
    Here, an equivalence class of matrices is defined as 
    \begin{align}
    	[\mathbf{R}] = \left\{\widetilde{\mathbf{R}} \in \mathcal{M}_{3} \left(\mathcal{D}_{\triangle x}^{1}\right) | \widetilde{\mathbf{R}} \sim \mathbf{R} \right\}, 
    \end{align}
    where 
    \begin{align}
    \widetilde{\mathbf{R}} \sim \mathbf{\mathbf{R}} 
    :\Longleftrightarrow
    \mathbf{R} = \mathbf{U} \widetilde{\mathbf{R}} \mathbf{V}, 
    \end{align}
    with \(\mathbf{U}, \mathbf{V} \in \mathrm{GL}_{3}(\mathcal{D}_{\triangle x}^{1})\).     
    It is left to investigate whether and how some or all of the above derived conditions can be reformulated in the form of an equivalence transformations from \(\widetilde{\mathbf{T}}\) to \(\mathbf{T}\) and thus from \(\widetilde{\mathbf{M}}\) to \(\mathbf{M}\). 
    However, assuming that we have general existence in the sense that any FDS can be written in the form \eqref{eq:finiteDiffDef}, general existence of nonunique LBSs in the equivalence class \([\mathbf{M}]\) would follow. 
\end{remark}

\begin{remark}[Trivial example]
    The choices of \(s_{2}, s_{3} \in (0, 2]\) can be made arbitrarily, however with increased computational effort which is planned for future work. 
    Notably, we would like to point out that our counterexample above is nontrivial in this range of relaxation frequencies in the following sense. 
    For the purpose of comparison, we provide a trivial counterexample as well. 
    Let $s_{2}=s_{3}=1$ and $s_{1} = 0$, such that \(\mathbf{S} = \mathrm{diag} ( 0, 1, 1) \in \mathcal{M}_{3}(\mathbb{R})\), and let \(\mathbf{T} = (t_{ij})_{ij} \in \mathcal{M}_{3} \left(\mathcal{D}_{\triangle x}^{1}\right)\) with $t_{ij} \in \mathcal{D}_{\triangle x}^{1} $ and $i,j \in \{1, 2, 3\}$.  
    Thus, we obtain
    \begin{align}
        \mathbf{A} 
        \coloneqq 
        \mathbf{T}(\mathbf{I}- \mathbf{S}) 
        =
        \begin{bmatrix}
            t_{11} & 0 & 0 \\
            t_{21} & 0 & 0 \\
            t_{31} & 0 & 0
        \end{bmatrix}, 
        \quad
        \mathbf{B} 
        \coloneqq
        \mathbf{T} \mathbf{S} 
        =
        \begin{bmatrix}
            0 & t_{12} & t_{13} \\
            0 & t_{22} & t_{23} \\
            0 & t_{32} & t_{33}
        \end{bmatrix},
    \end{align} where $\mathbf{A}, \mathbf{B} \in \mathcal{M}_{3} \left(\mathcal{D}_{\triangle x}^{1}\right)$.
    Since $s_{2}=s_{3}=1$, we relax on the equilibrium and have $\gamma_{1}= \gamma_{0}=0$, which yields
    \begin{align}
        m_{1}^{n+1} 
        = 
        & -\gamma_{2}m_{1}^{n} +  \big(\mathbf{B} \bm{m}^{\mathrm{eq}}|^{n}\big)_{1} + \big(\underbrace{(\gamma_{2}\mathbf{I}+\mathbf{A})\mathbf{B}}_{=0} \bm{m}^{\mathrm{eq}}|^{n-1}\big)_{1} \nonumber \\
        &  + \big(\underbrace{ (\gamma_{2}\mathbf{A} +  \mathbf{A}^{2})\mathbf{B}}_{=0}\bm{m}^{\mathrm{eq}}|^{n-2}\big)_{1} \\
        \label{Simplified moment equation}
        =
        & -\gamma_{2}m_{1}^{n} +  \mathbf{B} \bm{m}^{\mathrm{eq}}|_{1}^{n},
    \end{align} where $(\gamma_{2}\mathbf{I}+\mathbf{A})\mathbf{B}=0$ and $( \gamma_{2}\mathbf{A} +  \mathbf{A}^{2})\mathbf{B} = 0$ as a consequence of the Cayley--Hamilton theorem (see \cite{bellotti2022finite} and references therein). 
    When looking at \eqref{eq:counterFds}, we see that firstly, the matrices $\mathbf{A}$ and $\widetilde{\mathbf{A}}$ for two distinct LBSs leading to the same FDS must have the same characteristic polynomial. 
    Since we have
    \begin{align}
        \gamma_{2} = - \mathrm{tr}(\mathbf{A}) = - t_{11},
    \end{align}
    equation \eqref{Simplified moment equation} can be written as
    \begin{align}\label{simplified moment equation 2}
         m_{1}^{n+1} 
        & =
         t_{11}m_{1}^{n} + t_{12}m_{2}^{\mathrm{eq}}\vert^{n} + t_{13}m_{3}^{\mathrm{eq}}\vert^{n} .
    \end{align}
    This in turn means that the schemes will be equal if
    \begin{align}
        \label{Necessary Condition}
        t_{11} = \widetilde{t}_{11}, \quad  t_{12} = \widetilde{t}_{12}, \quad t_{13} = \widetilde{t}_{13}.
    \end{align}
    Therefore, the remaining entries of $\mathbf{A}$ and $\widetilde{\mathbf{A}}$ and $\mathbf{B}$ and $\widetilde{\mathbf{B}}$ respectively could differ, as long as $\mathbf{A}$ and $\widetilde{\mathbf{A}}$ fulfill the three conditions \eqref{Necessary Condition} (first rows of $\widetilde{\mathbf{T}}$ and $\mathbf{T}$ have to coincide), but these different matrices would still lead to the same FDS. 
    Since the LBSs resulting from \(\mathbf{T}\) and \(\widetilde{\mathbf{T}}\) and \(\mathbf{S}\) decouple the evolution equation for \(m_{1}\) from the higher order moment equations for \(m_{2}\) and \(m_{3}\), respectively, we regard this example as trivial. 
    As motivated by \cite{bellotti2022finite}, a comparison to systems of ordinary differential equations which can be arbitrarily extended with decoupled equations, underlines the triviality of decoupled moment equations in a relaxation system. 
    Note that, in the proof of Proposition~\ref{prop:fdsTolbs} we make a choice for \(\mathbf{S}\) which retains the dependence on \(m_{2}\) and \(m_{3}\) in \eqref{eq:reducedFDS} and can thus be regarded as nontrivial in this sense.     
\end{remark}

\begin{remark}[Particular LBS]\label{rem:explicitMoments}
    We further motivate the idea of using \([\mathbf{M}]\) proposed in Remark~\ref{rem:equivalence} to describe possible moment systems leading to the same FDSs with another nontrivial and more specific counterexample suitable for the proof of Proposition~\ref{prop:lbsToFds} on a smaller \(D_{1}Q_{2}\) velocity stencil. 
    Let \(s_{1}=0\), \(s_{2}=s\), and 
    \begin{align}
        \label{Moment Matrix Choice}
        \mathbf{M} = 
        \begin{bmatrix}
            1 & 1  \\
            1 & -1 
        \end{bmatrix} ,
        \qquad
        \widetilde{\mathbf{M}} = 
        \begin{bmatrix}
            \widetilde{m}_{11} & \widetilde{m}_{12}  \\
            \widetilde{m}_{21} & \widetilde{m}_{22}
        \end{bmatrix} ,
    \end{align}
    and assume that the equilibria are linear functions of the conserved moments, 
    \begin{align}
        m_{2}^{\mathrm{eq}} = \epsilon m_{1} , \qquad \widetilde{m}_{2}^{\mathrm{eq}} = \widetilde{\epsilon} \widetilde{m}_{1},  
    \end{align}
    respectively. 
    Using \cite[Equation (9.2)]{bellotti2023numerical}, we can see that the LBS matrices 
    \begin{align}
        \mathbf{E} = \mathbf{A} + \mathbf{B} \begin{pmatrix} \epsilon \\ \epsilon \end{pmatrix} \otimes \bm{e}_{1}, \qquad \widetilde{\mathbf{E}} = \widetilde{\mathbf{A}} + \widetilde{\mathbf{B}} \begin{pmatrix} \widetilde{\epsilon} \\ \widetilde{\epsilon} \end{pmatrix} \otimes \bm{e}_{1}, 
    \end{align}
    will produce the same FDS, if the characteristic polynomials are equal. 
    Thus, using \eqref{eq:Tmatrix} in Fourier space 
    \begin{align}
        \widehat{\mathbf{T}} = \mathbf{M} \mathrm{diag}\left( e^{-\mathsf{i} \theta}, e^{\mathsf{i} \theta} \right) \mathbf{M}^{-1},
    \end{align} 
    we compute \eqref{eq:Amatrix} and \eqref{eq:Bmatrix}, equate the characteristic polynomials \(\chi_{\widehat{\mathbf{E}}} = \chi_{\widehat{\widetilde{\mathbf{E}}}}\), assume \(\widetilde{\epsilon}=\epsilon\) for simplicity, solve the resulting system of equations, and finally obtain a condition 
    \begin{align}\label{eq:finalMcondition}
        \widetilde{m}_{11} = \frac{\widetilde{m}_{12} \widetilde{m}_{21} (\epsilon + 1)}{\epsilon \widetilde{m}_{22} + 2 \widetilde{m}_{12} \epsilon - \widetilde{m}_{22}}
    \end{align}   
    for fixed \(\widetilde{m}_{12}, \widetilde{m}_{21}, \widetilde{m}_{22}, \epsilon \). 
    Thus, we have constructed nonunique LBSs in \([\mathbf{M}]\) generated by \eqref{eq:finalMcondition} which all lead to the same FDS according to Proposition~\ref{prop:lbsToFds}. 
    In particular, for the $D_{1}Q_{2}$ velocity stencil, we have $c_{1} = 1, c_{2}=-1$ and in turn
\begin{align}\label{eq:Tdefin}
    \mathbf{T} = \mathbf{M} \mathrm{diag}\left( \mathsf{x}, \overline{\mathsf{x}} \right) \mathbf{M}^{-1},
\end{align}
where $\mathsf{x} \eqcolon \mathsf{T}_{\triangle x}^{c_{1}} = \mathsf{T}_{\triangle x}^{1}$ and $\overline{\mathsf{x}} \eqcolon \mathsf{T}_{\triangle x}^{c_{2}} = \mathsf{T}_{\triangle x}^{-1} $, see \cite{bellotti2022finite}. 
For the sake of clarity, we adopt the notation from \cite{bellotti2022finite}, where the definition of a generic shift operator for a given function $f$ and \(\bm{z} \in \mathbb{Z}^{d}\) is 
\begin{align}
    \left(\mathsf{T}_{\triangle x}^{\bm{z}}f \right)(\bm{x}) = f(\bm{x} - \bm{z} \triangle x).
\end{align}
In case that \(d=1\) and \(z=1\), we have 
\begin{align}
    \left(\mathsf{T}_{\triangle x}^{1}f\right)(x) & = f(x -  \triangle x), \\
    \left(\mathsf{T}_{\triangle x}^{-1}f\right)(x) & = f(x +  \triangle x). 
\end{align}
For the moment matrix $\mathbf{M}$ as in \eqref{Moment Matrix Choice}
we obtain 
\begin{align}
    \mathbf{T} 
    & = 
    \mathbf{M} 
    \begin{bmatrix}
        \mathsf{x} & 0 \\
        0 & \overline{\mathsf{x}}
    \end{bmatrix}
    \mathbf{M}^{-1}  \\
    & =
    \begin{bmatrix}
        1 & 1 \\
        1 & -1
    \end{bmatrix}
    \begin{bmatrix}
        \mathsf{x} & 0 \\
        0 & \overline{\mathsf{x}}
    \end{bmatrix} \frac{1}{2}
    \begin{bmatrix}
        1 & 1 \\
        1 & -1
    \end{bmatrix} \\
    & = \frac{1}{2}
    \begin{bmatrix}
       \mathsf{x} + \overline{\mathsf{x}} & \mathsf{x} - \overline{\mathsf{x}} \\
        \mathsf{x} - \overline{\mathsf{x}} & \mathsf{x} + \overline{\mathsf{x}}
    \end{bmatrix} . 
\end{align}
For the moment matrix $\mathbf{\widetilde{M}}$ as in \eqref{Moment Matrix Choice}, where $\widetilde{m}_{11}\widetilde{m}_{22} \neq \widetilde{m}_{12}\widetilde{m}_{21}$, we obtain an explicit expression for $\widetilde{\mathbf{T}}$ as well, 
\begin{align}
    \widetilde{\mathbf{T}} 
    &= 
    \widetilde{\mathbf{M}}
    \begin{bmatrix}
        \mathsf{x} & 0 \\
        0 & \overline{\mathsf{x}}
    \end{bmatrix}
    \widetilde{\mathbf{M}}^{-1} \\
    &= 
    \begin{bmatrix}
        \widetilde{m}_{11} & \widetilde{m}_{12} \\
        \widetilde{m}_{21} & \widetilde{m}_{22}
    \end{bmatrix}
    \begin{bmatrix}
        \mathsf{x} & 0 \\
        0 & \overline{\mathsf{x}}
    \end{bmatrix} 
    \frac{1}{\widetilde{m}_{11}\widetilde{m}_{22}-\widetilde{m}_{12}\widetilde{m}_{21}}
    \begin{bmatrix}
        \widetilde{m}_{22} & -\widetilde{m}_{12} \\
        -\widetilde{m}_{21} & \widetilde{m}_{11}
    \end{bmatrix} \\
    &= \frac{1}{\widetilde{m}_{11}\widetilde{m}_{22}-\widetilde{m}_{12}\widetilde{m}_{21}}
    \begin{bmatrix}
        \widetilde{m}_{11}\mathsf{x} & \widetilde{m}_{12}\overline{\mathsf{x}} \\
        \widetilde{m}_{21}\mathsf{x} & \widetilde{m}_{22}\overline{\mathsf{x}}
    \end{bmatrix}
    \begin{bmatrix}
        \widetilde{m}_{22} & -\widetilde{m}_{12} \\
        -\widetilde{m}_{21} & \widetilde{m}_{11}
    \end{bmatrix}
    \\
    &= \frac{1}{\widetilde{m}_{11}\widetilde{m}_{22}-\widetilde{m}_{12}\widetilde{m}_{21}}
    \begin{bmatrix}
        \widetilde{m}_{11}\widetilde{m}_{22}\mathsf{x} -  \widetilde{m}_{12}\widetilde{m}_{21}\overline{\mathsf{x}}
        &
        -\widetilde{m}_{11}\widetilde{m}_{12}\mathsf{x} +  \widetilde{m}_{12}\widetilde{m}_{11}\overline{\mathsf{x}}
        \\
        \widetilde{m}_{21}\widetilde{m}_{22}\mathsf{x} -  \widetilde{m}_{22}\widetilde{m}_{21}\overline{\mathsf{x}}
        &
        -\widetilde{m}_{21}\widetilde{m}_{12}\mathsf{x} +  \widetilde{m}_{22}\widetilde{m}_{11}\overline{\mathsf{x}}
    \end{bmatrix}
\end{align}
Notably, further equations for explicitly computing the components of \(\widetilde{\mathbf{M}}\) can be derived with the help of $\mathbf{A}$, $\mathbf{B}$ and \(\widetilde{\mathbf{A}}\), \(\widetilde{\mathbf{B}}\) for \(\mathbf{T}\) and \(\widetilde{\mathbf{T}}\), respectively. 
From equating the resulting FDSs, we obtain similar conditions as in the proof of Proposition~\ref{prop:fdsTolbs}: equal characteristic polynomials of \(\mathbf{A}\) and \(\widetilde{\mathbf{A}}\), coinciding first rows of \(\mathbf{B}\) and \(\widetilde{\mathbf{B}}\), coinciding first rows of \(\mathbf{A}\mathbf{B}\) and \(\widetilde{\mathbf{A}}\widetilde{\mathbf{B}}\). 
Based on the resulting three equations, and the condition \eqref{eq:finalMcondition}, we hence are able to explicitly construct two matrices $\mathbf{T}$ and $\widetilde{\mathbf{T}}$ with the same FDS but differing corresponding LBSs. 
\end{remark}

\section{Conclusions}\label{sec:conclusions}
In this article, we have proven that, given an FDS of the form \eqref{eq:finiteDiffDef}, its resulting LBS is not unique (cf.\ Proposition~\ref{prop:fdsTolbs}). 
Our observation is based on specific counterexamples, which motivate the definition of equivalence classes of moment matrices used in LBSs. 
These classes could further categorize the nonunique relaxation systems in the construction approach proposed by Simonis \textit{et al.}~\cite{simonis2020relaxation,simonis2022constructing,simonis2023lattice}. 
In future studies, we will investigate the construction of specific equivalence classes of relaxation systems for lattice Boltzmann methods derived from finite difference methods. 
To classify our present findings we challenge the notion of equivalence used in this work. 
To state whether two numerical methods are equivalent (or not), we distinguish between two mapping types of LBS equations. 
The first approach is based on mappings between schemes by using matrix-based transformations, which we primarily consider in this paper. 
The second approach to be mentioned is based on changes of variables of the schemes by using substitutions. 
We do not consider these reformulations within the present work and instead refer the interested reader for example to the work of He \textit{et al.}~\cite{he1998novel} where a numerical integration along characteristics and a trapezoidal rule lead to an implicit scheme that is made explicit by a change of variables, yielding a classical LBS. 
At last, we would like to point out that the possibility of deriving a class of LBSs from a specific FDS, where all schemes in this class are linked by changes of variables, still persists. 
Hence, to prove that the LBSs in this class are different also from this perspective, in future studies it would be necessary to show that no variable changes, transforming the LBSs into one another, exist. 
Besides, using numerical experiments should be considered in future work to underline the theoretical findings.

\section*{Acknowledgments}
The authors thank T.\ Bellotti for valuable discussions and comments. 
S.\ Simonis acknowledges support from KHYS at KIT through a Networking Grant.

\section*{Author contribution}
\textbf{E.\ Kummer}: 
Conceptualization, 
Methodology, 
Validation, 
Formal analysis, 
Investigation, 
Writing - Original Draft, 
Writing - Review \& Editing; 
\textbf{S.\ Simonis}: 
Conceptualization, 
Methodology, 
Validation, 
Formal analysis, 
Investigation, 
Resources, 
Writing - Original Draft, 
Writing - Review \& Editing, 
Supervision, 
Project administration, 
Funding acquisition.
All authors have read and approved the final version of this manuscript.



\end{document}